\documentclass[11pt,a4paper]{article}%
\DeclareMathAlphabet{\mathcal}{OMS}{cmsy}{m}{n}
\usepackage[latin1]{inputenc}
\usepackage[T1]{fontenc}
\usepackage[english]{babel}

\usepackage{comment}
\usepackage{graphicx}
\usepackage{caption}
\usepackage{subcaption}

\usepackage{indentfirst}
\usepackage[
allcolors=blue,colorlinks=true, pdfstartview=FitH, linkcolor=blue, pdftoolbar=true, bookmarks=true,bookmarksnumbered,plainpages,backref=true]{hyperref}
\usepackage[lmargin=3cm,tmargin=2.5cm,bmargin=2.5cm,rmargin=3cm]{geometry}
\usepackage{amsmath,amssymb,amsthm,amsfonts}
\usepackage{mathtools}
\usepackage{times}

\def\Xint#1{\mathchoice
	{\XXint\displaystyle\textstyle{#1}}%
	{\XXint\textstyle\scriptstyle{#1}}%
	{\XXint\scriptstyle\scriptscriptstyle{#1}}%
	{\XXint\scriptscriptstyle\scriptscriptstyle{#1}}%
	\!\int}

\def\XXint#1#2#3{{\setbox0=\hbox{$#1{#2#3}{\int}$}
		\vcenter{\hbox{$#2#3$}}\kern-.5\wd0}}

\newcommand{\fin}{\hfill \ensuremath{\square}}
\newcommand{\Csharp}{{\settoheight{\dimen0}{C}\kern-.09em \resizebox{!}{\dimen0}{\raisebox{\depth}{$\sharp$}}}}

\newtheoremstyle{theorem}
{5pt +1\p@ -2.0\p@}
{5pt +1\p@ -2.0\p@}
{\it}			      
{}				  
{\bfseries}   
{.}               
{.4em}       
{}
\theoremstyle{theorem}
\newtheorem{theorem}{Theorem}[section]

\newtheorem{corollary}[theorem]{Corollary}
\newtheorem{lemma}[theorem]{Lemma}
\newtheorem{remark}[theorem]{Remark}
\numberwithin{equation}{section}

\usepackage[scaled=1.0576331,helvratio=0.9160502]{newtxtext}
\usepackage{amsmath,amssymb,amsthm,amsfonts,mathtools}
\usepackage{mathptmx}
\usepackage{dsfont} 
\usepackage{bbm}
\usepackage{pifont}

\usepackage{enumitem}

\begin{document}
\title{\Large{Nonlinear boundary problem for Harmonic functions in higher dimensional Euclidean half-spaces}}
\author{{Marcelo F. de Almeida}{\thanks{%
			de Almeida, M.F. was supported by CNPq:409306/2016, Brazil.}} \\
	{\small Universidade Federal de Sergipe, Departamento de Matem\'atica,} \\
	{\small {\ CEP 49100-000, Aracaju-SE, Brazil.}}\\
	{\small \texttt{Email:marcelo@mat.ufs.br}}\vspace{0.5cm}\\
	{Lidiane S. M. Lima}{\thanks{%
			Lima, L.S.M. was supported by CNPq:409306/2016, Brazil (Corresponding author).}}\\
	{\small Universidade Federal de Goi\'as, IME - Departamento de Matem\'atica,}
	\\
	{\small CEP 740001-970, Goiania-GO, Brazil.} \\
	{\small \texttt{Email:lidianesantos@ufg.br}}}
\date{}
\maketitle
\begin{abstract}
In this paper we are interested on solvability of the problem 
\begin{align*}
\begin{cases}
-\Delta u=0 & \text{in} \;\;\;\mathbb{R}^{n+1}_{+}\;\;\;\;\;\;\;\;\;\\
\;\;\displaystyle{\frac{\partial u}{\partial \nu}} = V(x)u+b \vert u\vert^{\rho-1}u+f \; & \text{on} \;\;\partial\mathbb{R}^{n+1}_+\;\;\;\;\;\;\;\;\,\,
\end{cases}
\end{align*}
with high singular data $f$ and potential $V$ on boundary $\partial\mathbb{R}^{n+1}_+$ of half-space $ \mathbb{R}^{n+1}_{+}=\{(x,t)\in\mathbb{R}^{n+1}\,\vert\, t>0\}$ for $n\geq 2$. More precisely, inspired at \cite{deAlmeida1} and \cite{Quittner} we introduce a new functional space based in weak-Morrey spaces and we shown  existence of positive solutions $u$ to the above problem when inhomogeneous term   $f\in\text{weak-}\mathcal{M}_{p}^{n{(\rho-1)}/{\rho}}(\mathbb{R}^{n})$  and potential $V\in   \text{week-}\mathcal{M}^{n}_{\ell}(\mathbb{R}^{n})$ are sufficiently small in the natural $n/(n-1)<\rho<\infty$. Our theorems recover the   range $(n+1)/(n-1)\leq \rho<\infty$  and immediately imply in solvability of the equivalent nonlocal half-Laplacian problem  $(-\Delta)^{{1}/{2}}u=Vu+b\vert u\vert^{\rho-1}u+ f (x)$ for $f$ and potential $V$ rough than previous ones, in view of strictly inclusions  $L^\lambda \varsubsetneq\mathcal{M}^{\lambda}_{p} \varsubsetneq \text{week-}\mathcal{M}^{\lambda}_{p}$ for $1<p<\lambda<\infty$. Also,  from Campanato's lemma we conclude that $u\in C^{0,\alpha}_{loc}( \overline{\mathbb{R}^{n+1}_+})$ is locally H\"older continuous, for  $f\in\mathcal{M}_{p}^{n{(\rho-1)}/{\rho}}(\mathbb{R}^{n})$ and $V\in \mathcal{M}^{n}_{\ell}(\mathbb{R}^{n})$ in Morrey spaces.

\bigskip\noindent\textbf{AMS MSC:} 49J52, 35J05, 35J65, 35J67

\medskip\noindent\textbf{Keywords:} Morrey spaces, harmonic function, Neumann boundary, layer potentials. 
\end{abstract}


\setcounter{equation}{0}\setcounter{theorem}{0}

\section{Introduction}
In this paper we are interested in well-posedness of the nonlinear boundary problem
\begin{align}\label{Neumann}
\begin{cases}
\Delta u=0 & \text{in} \;\;\;\mathbb{R}^{n+1}_{+}\\
\displaystyle{\partial_\nu u} = V(x)u+b \vert u\vert^{\rho-1}u+f \; & \text{on} \;\;\partial\mathbb{R}^{n+1}_+,
\end{cases}
\end{align}
when the force $f$ and potential  $V$ are taken in weak-Morrey spaces or Morrey spaces on boundary $\partial \mathbb{R}^{n+1}_{+}$ of half-space $ \mathbb{R}^{n+1}_{+}=\{(x,t)\in\mathbb{R}^{n+1}\,\vert\, t>0\}$, for all $n/(n-1)<\rho<\infty$ with $n\geq 2$ and $b\geq 0$. Here  ${\partial_\nu u}=-\partial_{t} u$ denotes the exterior normal derivative on boundary $\partial \mathbb{R}^{n+1}_{+}$. The linear problem associated to \eqref{Neumann}  read as follows
\begin{equation}\label{Linear}
\Delta u=0 \;\; \text{in} \;\;\;\mathbb{R}^{n+1}_{+} \;\;\text{ and } \;\;\partial_\nu u = V(x)u +f\;\;  \text{on}\;\;\partial\mathbb{R}^{n+1}_+.
\end{equation}
This problem has been studied in many papers (see e.g. \cite{Armitage, Gardiner}) and can be rewritten as a single layer potential
\begin{equation}\label{lay-pot}
u(x,t)=(\mathbf{S}\tilde{f})(x,t):=\frac{\Gamma((n+1)/2)}{\pi^{(n+1)/2}}\int_{\partial\mathbb{R}^{n+1}_+} \frac{1}{(|x-y|^2+t^2)^{\frac{n-1}{2}}}\tilde{f}(y)dy
\end{equation}
for  $\tilde{f}(y)=V(y)u +f(y)$ such that $\int_{\mathbb{R}^n}(1+|y|^2)^{-\frac{n}{2}}\vert \tilde{f}(y)\vert dy<\infty$. The last integral can be  pointwise controled by fractional Hardy-Littlewood maximal function
\begin{equation}
(M_{\alpha}\tilde{f})(y_0)\,=\sup_{r>0,\,y_0\in\,\mathbb{R}^{n}} r^{\alpha-n} \int_{B(y_0,r)}\vert \tilde{f}(y)\vert dy,\quad 0<\alpha<n\nonumber 
\end{equation}
since, in view of $(1+|y-y_0|)^2\leq (1+|y_0|)^2\,(1+|y|^2)$, we have 
\begin{align}
\int_{\mathbb{R}^{n}}\frac{|\tilde{f}(y)|}{(1+|y|^2)^{\frac{n}{2}}}dy&\lesssim \int_{\mathbb{R}^{n}}\frac{|\tilde{f}(y)|}{1+|y-y_0|^{n}}dy \nonumber\\
&=\int_{B(y_0,1)}\frac{\vert \tilde{f}(y)\vert}{1+\vert y-y_0\vert^{n}}dy +\sum_{k=1}^{\infty}\int_{\{y\;\vert\; 2^{k-1}\leq \vert y-y_0\vert < 2^k\}}\frac{|\tilde{f}(y)|}{1+|y-y_0|^{n}}dy\nonumber\\
&\lesssim \frac{1}{\vert B(y_0,1)\vert }\int_{B(y_0,1)}|\tilde{f}(y)|dy + 2^{n}\sum_{k=1}^{\infty} 2^{-k\alpha} \,2^{k(\alpha -n)}\int_{B(y_0,2^k)}|\tilde{f}(y)|dy\nonumber\\
&\lesssim (M_0\tilde{f})(y_0)+(M_{\alpha}\tilde{f})(y_0)\quad \text{ for } \quad\alpha>0\nonumber.
\end{align}
Hence, to study \eqref{Neumann} and  \eqref{Linear} with external force $f$ and singular potential $V$ in  rough function spaces it seems interesting verify the continuously of  maximal operators $M_{\alpha}$ in order to control the boundary data on rough spaces. This was the first motivation to study these problems with $f$ in weak-Morrey space $\mathcal{M}^\omega_{p\infty}(\mathbb{R}^{n})$ and $V\in\mathcal{M}^{n}_{\ell\infty}(\mathbb{R}^{n})$ in view of continuously  of $M_{\alpha}$ in Morrey-Lorentz spaces (see \cite[Corollary 1.3]{Marcelo-Lidiane}). 
The second motivation is that \eqref{Neumann} 
can be seen as a local realisation of square root of Laplacian (or half-Laplacian) problem 
\begin{align}\label{nonlocal1}
(-\Delta)^{{1}/{2}}v(x)=V(x)v+b\vert v(x)\vert^{\rho-1}v(x)+ f (x)\quad \text{ in }\quad \mathbb{R}^{n},\quad b>0
\end{align}
via Dirichlet to Neumann map  $T:v\mapsto \textnormal{tr}_{\partial\mathbb{R}^{n+1}_+}\partial_{\nu} u$, where $u(x,t)=(P_{t}\ast v)(x)$ is the harmonic extension of  $v$  and $P_t$ denotes the Poisson kernel on half-space $\mathbb{R}^{n+1}_+$
\begin{equation}
P_t(x)=\frac{\Gamma((n+1)/2)}{\pi^{(n+1)/2}}\frac{t}{(|x|^2+t^2)^{(n+1)/2}}\nonumber\quad x\in \mathbb{R}^n,\; t>0.
\end{equation}
Indeed, the Dirichlet to Neumann map $T$ is the half-Laplacian  $(-\Delta_x)^{1/2}$ in $\mathbb{R}^{n}$  since from $\widehat{P_t}(\xi)=\exp\{-t\vert\xi\vert\}$ we have
\begin{equation}
\textnormal{tr}_{\partial\mathbb{R}^{n+1}_+} (-\partial_{t} u){\,\widehat{}}\,(\xi,t)= \textnormal{tr}_{\partial\mathbb{R}^{n+1}_+} [-\partial_{t} {\,\widehat{u}}\,(\xi,t)]=\left(\vert\xi\vert e^{-t\vert\xi\vert}\widehat{v}(\xi)\right)\Big\lvert_{t=0} =\vert\xi\vert \widehat{v}(\xi)=[(-\Delta_x)^{{1}/{2}} v ]{\,\widehat{}}\,(\xi)\nonumber,
\end{equation}
see Caffarelli-Silvestre \cite{CF} for generalization. Thus, if $u(x,t)$ is a solution of \eqref{Neumann} clearly its trace $v(x):= \textnormal{tr}_{\partial\mathbb{R}^{n+1}_+} u(x,t)$  $a.e.$ in $\mathbb{R}^n$ is a solution of \eqref{nonlocal1}. 
Therefore, to show existence of solutions to \eqref{nonlocal1} via its local realisation with inhomegeneous term $f$ in weak-Morrey spaces we need consider a Banach space $X_{r,q}^{1,s_r,s_q}(\Omega)$ with suitable information on the integrability 
of trace of  $u$ on boundary $\partial\Omega$ and integrability of  $\vert \nabla u \vert$ on $\Omega\subset \mathbb{R}^{n+1}$. This Banach space will be chosen as a substitute for Sobolev type spaces $W^1X_r^{s_r}(\Omega)$ 
in domains $\Omega$ with regular or irregular boundary for which is not known  characterizations of trace space on $\partial\Omega$ for functions $u$ such that $ \Vert \nabla u\Vert_{X_{r}^{s_r}(\Omega)}<\infty$. For instance, consider the special situation where $X_{r,q}^{1,s_r,s_q}(\Omega)$ is given by 
\begin{equation}\nonumber
	\Vert u\Vert_{X_{r,q}^{1}(\Omega)}= \Vert \nabla u\Vert_{L^r(\Omega)} +  \Vert \textnormal{tr}_{\partial\Omega}u\Vert_{L^{q}(\partial\Omega)}, \quad q=r{^\sharp}={rn}/{[(n+1)-r]}
\end{equation}
for any open set $\Omega$ in $\mathbb{R}^{n+1}$ endowed by a measure $\mu$ satisfying $ \mu (B_r(x)\cap \Omega)\leq c\,r^{\alpha}\,$ without any regularity condition on boundary $\partial\Omega$. 
Maz'ya and Cianchi \cite[Section 2]{Mazya2} shown, in particular, there exists positive constants $c_j$ such that 
\begin{equation}\label{Mazya}
	\Vert u\Vert_{L^{r^{\ast}_{\alpha}}(\Omega)}\leq c_1 \Vert \nabla u\Vert_{L^r(\Omega)} +  c_2\Vert \textnormal{tr}_{\partial\Omega}u\Vert_{L^{r^\sharp}(\partial\Omega)},\quad r^\ast_{\alpha}={r\alpha}/{[(n+1)-r]}
\end{equation}
for every continuous functions on $\overline{\Omega}$ such that $u\in W^{1,r}(\Omega)$, for all $1<r<n+1$  and $\alpha\in (n,n+1]$.  Maz'ya and Cianchi \cite[Theorem 6.1]{Mazya2} established \eqref{Mazya} for higher order derivative in rearrangement-invariant spaces for which its representation on $(0,\infty)$ 
satisfies one-dimensional Hardy type inequalities \cite[Section 1.3.2]{Mazya1}. 
The ``Sobolev trace type''  inequality  \eqref{Mazya} 
was first established by Maz'ya \cite{Mazya3} via isoperimetric inequalities, where he also exhibited optimal constants as  $r=1$. In \cite{Maggi1,Maggi2} employing optimal transportation theory the authors exhibited  optimal constants to  \eqref{Mazya} for all $1<r\leq n+1$ when $\Omega$ is locally Lipschitz open domain. As we can see, the Banach space $X_{r,q}^{1,s_r,s_q}(\Omega)$ provide a good substitute for Sobolev spaces  $W^1X_r^{s_r}(\Omega)$ when the trace space  $TW^1X_r^{s_r}(\Omega)$ is not known for rearrangement-invariant spaces $X_r^{s_r}(\Omega)$ in irregular or regular domains (see also \cite[p.1912]{Ferreira3} for a discussion). The lack of  boundary trace in Morrey or Morrey-Lorentz spaces (see also \cite{deAlmeida1} in weak-$L^p(\mathbb{R}^n)$  and \cite{Quittner} in  $L^1(\Omega)$) 
inspired us  to  study solvability of \eqref{Neumann} in the space   ${X}^{1,s_r,s_q}_{r,q}(\mathbb{R}^{n+1}_{+})$  of locally integrable functions $u\in L^1_{loc}(\overline{\mathbb{R}^{n+1}_+})$ such that 
\begin{equation}\label{Functional-space}
\Vert u \Vert_{{X}^{1,s_r,s_q}_{rq}(\mathbb{R}^{n+1}_{+})}=\Vert\, \nabla u \, \Vert_{\mathcal{M}^{\mu}_{r,s_r}(\mathbb{R}^{n+1}_+)}+\Vert \text{tr}_{\partial\mathbb{R}^{n+1}_+}u\Vert_{\mathcal{M}_{q,s_q}^{\lambda}(\partial \mathbb{R}^{n+1}_+)}<\infty,
\end{equation}
where $\mathcal{M}^\mu_{rs}(\mathbb{R}^{n+1}_+)$ denotes the Morrey-Lorentz space endowed by norm
\begin{equation}
	\Vert u \Vert_{\mathcal{M}^\mu_{rs}}=\sup_{R>0, (x,t)\in \mathbb{R}^{n+1}_+}\vert \mathbb{R}^{n+1}_+\cap B_{R}(x,t)\vert^{\frac{1}{\mu}-\frac{1}{r}}\Vert u\Vert_{L^{qs}(\mathbb{R}^{n+1}_+\cap B_{R})}\nonumber
\end{equation}
for every $1<r\leq\mu<\infty$ and $1\leq s\leq \infty$. Note that $\mathcal{M}^\mu_{rs}(\mathbb{R}^{n+1}_+)$ is  rearrangement-invariant, since   
$$\Vert \phi\Vert_{L^{qs}(\mathbb{R}^{n+1}_{+})} := \left\Vert t^{\frac{1}{q}-\frac{1}{s}}\phi_m^{\ast}(t)\right\Vert_{L^s(0,\infty)}
= \Vert \psi\Vert_{L^{qs}(\mathbb{R}^{n+1}_{+})}$$
whenever $\phi_m^{\ast}=\psi_m^{\ast}$, where $\phi_m^{\ast}(t)=\inf\{ s\geq 0\,:\, m(\{\vert \phi\vert >s\})\leq t \}$ for $t\in [0,\infty)$. Moreover,  ${X}^{1,s_r,s_q}_{r,q}(\mathbb{R}^{n+1}_{+})$ endowed by norm $\Vert\cdot\Vert_{{X}^{1,s_r,s_q}_{r,q}}$ is a Banach space. 
Our main result is addressed to solvability of  \eqref{Neumann} understood as an integral equation 
\begin{align}\nonumber
	u(x,t)&= c_n \int_{\partial\mathbb{R}^{n+1}_+}(\vert x-y\vert^2+t^2)^{\frac{n-1}{2}}(b \vert u\vert^{\rho-1}u+Vu+f)(y)dy.
\end{align}
\bigskip
\bigskip
\bigskip
\begin{theorem}[Main theorem]\label{existence} Let $1<p< \omega=n(\rho-1)/\rho$ for  $\frac{n}{n-1}<\rho<\infty$ and $n\geq2$. Let $b>0$ and $V\in \mathcal{M}_{\ell\infty}^{n}(\mathbb{R}^n)$ for $1<\ell< n$.
	\begin{itemize}
		\item[(I)] \textnormal{(Existence and uniqueness)} There are $\varepsilon>0$ and  $C>0$ such that if $\Vert f\Vert_{\mathcal{M}_{p\infty}^{\omega}(\mathbb{R}^n)}\leq \varepsilon/C$, the problem \eqref{Neumann} has a unique solution $u\in X_{rq}^{1,\infty,\infty}(\mathbb{R}^{n+1}_+)$ such that
	\begin{equation}
		\Vert \nabla u \Vert_{\mathcal{M}^{\mu}_{r\infty}(\mathbb{R}^{n+1}_+)}\leq c\prime\Vert f\Vert_{\mathcal{M}_{p\infty}^{\omega}(\mathbb{R}^{n})} \quad \text{ and }\quad \Vert \textnormal{tr}_{\partial\mathbb{R}^{n+1}_+}u\Vert_{\mathcal{M}_{q\infty}^{\lambda}(\partial \mathbb{R}^{n+1}_+)}\leq c\prime\Vert f\Vert_{\mathcal{M}_{p\infty}^{\omega}(\mathbb{R}^{n})}\nonumber
	\end{equation}
with exactly constant $c\prime= C\Vert f\Vert_{\mathcal{M}^{\omega}_{p\infty}(\mathbb{R}^n)}/\big(1-c\Vert V\Vert_{\mathcal{M}_{\ell \infty}^{n}(\mathbb{R}^n)}\big)>0$, for all $1<r<\mu<\lambda=n(\rho-1)$ and $\rho<q<\lambda$ such that $r/\mu\leq q/\lambda$ and $\rho/(\rho-1)=n/\mu+1/r$.

		\item[(II)] \textnormal{(Stability of data)}  The
		solution $u\in X_{rq}^{1,\infty,\infty}(\mathbb{R}^{n+1}_+)$ depends continuously of  singular function $f\in \mathcal{M}^{\omega}_{p\infty}(\mathbb{R}^n)$ and potential $V\in  \mathcal{M}_{\ell\infty}^{n}(\mathbb{R}^n)$.
	\end{itemize}	
\end{theorem}

Some comments are in order. Firstly, as we seen above from Dirichlet to Neumann map $v\mapsto \text{tr}_{\partial\mathbb{R}^{n+1}}\partial_{\nu}u$ we conclude that \eqref{Neumann} is the local realisation to half-Laplacian problem \eqref{nonlocal1} whose solutions  $v=\text{tr}_{\partial\mathbb{R}^{n+1}_+}u$ are trace on $\partial\mathbb{R}^{n+1}_+$ of solutions $u(x,t)$ of \eqref{Neumann}. Hence, we immediately get the following Corollary:
\begin{corollary}\label{nonlocal2}For $n\geq 2$ and $n/(n-1)<\rho<\infty$, let $1<p<\omega=n(\rho-1)/\rho$ and $V\in \mathcal{M}_{\ell\infty}^{n}(\mathbb{R}^n)$ for all $1<\ell<n$. There are $\varepsilon>0$ and  $C>0$ such that if $\Vert f\Vert_{\mathcal{M}_{p\infty}^{\omega}(\mathbb{R}^n)}\leq \varepsilon/C$, the half-Laplacian problem \eqref{nonlocal1} has a solution $v
	$  
such that
		\begin{equation}
		 \Vert v\Vert_{\mathcal{M}_{q\infty}^{\lambda}(\mathbb{R}^{n})}\leq c\prime\Vert f\Vert_{\mathcal{M}_{p\infty}^{\omega}(\mathbb{R}^{n})}\nonumber
		\end{equation}
		where $c\prime= C\Vert f\Vert_{\mathcal{M}^{\omega}_{p\infty}(\mathbb{R}^n)}/\big(1-c\Vert V\Vert_{\mathcal{M}_{\ell \infty}^{n}(\mathbb{R}^n)}\big)>0$, for all $1<r<\mu<\lambda=n(\rho-1)$ and $\rho<q<\lambda$ such that $r/\mu\leq q/\lambda$ and $\rho/(\rho-1)=n/\mu+1/r$.
\end{corollary}

The solvability of \eqref{nonlocal1} has been investigated \cite{Ferreira5}, in particular, when $f$ and $V$ satisfies the Fourier transform decay 
$$\vert \widehat{f}(\xi)\vert\leq C\varepsilon \vert \xi\vert^{-[(n-1)-1/(\rho-1)]} \quad \text{and}\quad \vert \widehat{V}(\xi)\vert \leq L^{-1}\vert\xi\vert^{-(n-1)}$$
whenever $\rho\in\mathbb{N}$ and $\rho>n/(n-1)$ for $\varepsilon>0$ sufficiently small.  According to \cite[Remark 2.1]{deAlmeida2} there is \ $h\in \mathcal{M}_q^{\lambda}(\mathbb{R}^n)\subset \mathcal{M}^\omega_{p\infty}(\mathbb{R}^n)$ such that $ \sup_{\xi\in \mathbb{R}^n}\vert \xi\vert^{n(1-\frac{1}{\lambda})}\vert \widehat{h}(\xi)\vert = \infty $ whenever $1\leq q<\lambda<\infty$.  Here $\mathcal{M}_p^{\lambda}(\mathbb{R}^n)$ denotes the Morrey space defined by  $$\mathcal{M}_q^{\lambda}(\mathbb{R}^n)=\Big\{ g\in L^1_{loc}(\mathbb{R}^n)\,:\,  \int_{B_R(x)}\vert g\vert^q dx\leq c\, R^{-n({q}/{\lambda}-1)} \Big\}.$$
Hence, the inhomogeneous term $f$ and potential $V$ are rougher than those taken in pseudo-measure spaces  \cite{Ferreira5}.
The Corollary \ref{nonlocal2} provide a new class of functions $f\in \mathcal{M}_{p\infty}^{n{(\rho-1)}/{\rho}}(\mathbb{R}^{n})$ and $V\in \mathcal{M}_{\ell \infty}^{n}(\mathbb{R}^{n})$ for solvability of \eqref{nonlocal1}. 
Let us remark that  $\rho>0$  in Corollary \ref{nonlocal2} is a real number, but in \cite[Theorem 1.1]{Ferreira5} is integer due to Fourier approach. This technique was also employed in \cite{Lucas-Nestor} for nonlinearity with derivatives.  The results listed above does not cover the range $1<\rho\leq {n}/{(n-1)}$ which is expected for nonexistence of positive solutions as we will mention below. For $1\leq \rho<\frac{n+1}{n-1}$ is well known from \cite{Hu,Li-Zhang, Ou} there exists no solutions to the problem 
\begin{align}\label{Neumann2}
	\begin{cases}
		\Delta u=0 & \text{in} \;\;\;\mathbb{R}^{n+1}_{+},\\
		{\partial_{\nu} u}=bu^{\rho} \; & \text{on} \;\;\partial\mathbb{R}^{n+1}_+\\
		u>0 & \text{in} \;\;\;\mathbb{R}^{n+1}_{+},\\
	\end{cases}
\end{align}
or its equivalent nonlocal version $(-\Delta)^{\frac{1}{2}}v=bv^{\rho},\;\; v>0$ in $\mathbb{R}^{n}$, for all $n\geq2$.  However, if we add a potential $V$ on boundary, namely, $\partial_{\nu}u =u^{\rho}+V(x)u$ on $\partial\mathbb{R}^{n+1}_+$ then the problem \eqref{Neumann2} has at least a solution for $1< \rho<\frac{n+1}{n-1}$  and $V(x)=-1$. In other words,  the addiction of a potential  $V$ on boundary $\partial\mathbb{R}^{n+1}_+$ interfere the existence of positive solutions as was observed by Abreu, do Ó and Medeiros in \cite{Abreu} and Quittner-Reichel \cite[Theorem 17]{Quittner} for bounded domain $\Omega\subset \mathbb{R}^{n+1}_+$. This also happen for $\rho=\frac{n+1}{n-1}$ and $V(x)=-\alpha/\vert x\vert^\ell$, the authors \cite{Ferreira2} shown that existence or nonexistence of positive solutions is linked to homogeneity degree $\ell$ of $V(x)=-\alpha/\vert x\vert^\ell$, namely, $\ell=1$ or $\ell\neq1$ respectively. This is remarkable, since for  $\rho=(n+1)/(n-1)$ all solutions of \eqref{Neumann2} has the form (see \cite{Ou, Li-Zhang})  
\begin{equation}\label{positive-sol}
u(x,t)=\left(\frac{\varepsilon}{\vert x-\bar{x}\vert^{2}+\vert t-\bar{t}\vert^2}\right)^{\frac{n-1}{2}} \;\text{ where }\;\;  \varepsilon=\frac{-(n-1)\bar{t}}{b}\; \text{ for }\; b>0 \;\text{ and }\;(\bar{x},\bar{t})\in\mathbb{R}^{n+1}_{-}.\nonumber
\end{equation}
In our Theorem \ref{existence} or Corollary \ref{nonlocal2} we can take potential $V(x)=a(\frac{x}{\vert x\vert})\frac{c}{\vert x\vert}\in\mathcal{M}_{\ell}^{n}(\mathbb{R}^n)$ with $a\in L^{\infty}(\mathbb{S}^{n-1})$ nonnegative and  $c\geq 0$. Hence, we shown existence of positive solutions to \eqref{Neumann} in the range $n/(n-1)<\rho<\infty$ when 
\begin{equation}
(\mathbf{S}f)(x,t)=\int_{\partial \mathbb{R}^{n+1}_{+}}G(x-y,t)f(y)dy>0 \quad  \text{ in } \quad \mathbb{R}^{n+1}_{+}.\nonumber
\end{equation}
More precisely, we obtain the following theorem.
\begin{theorem}\label{Symmetries}Let  $\mathcal{D}\subset \mathbb{R}
	^{n}$ be a positive-measure set and let $f$ and  $V$ be according to Theorem \ref{existence}. If $f$ and  $V$ are nonnegative on $\partial\mathbb{R}
		^{n+1}_{+}$  and $f$ is positive on $\mathcal{D}$, then \eqref{Neumann} or \eqref{nonlocal1} has a positive solution.
\end{theorem}
The range ${n}/{(n-1)}<\rho<\infty$ and size control $\Vert f\Vert_{\mathcal{M}_{p\infty}^{n(\rho-1)/\rho}}$ is expected to show existence of positive solutions to \eqref{Neumann}. Indeed, according to Bernand \cite{Bernard} the  inhomogeneous equation 
\begin{align}\label{Bernard}
	\begin{cases}
		\Delta u + u^p +f(x)=0  &\text{in} \quad \mathbb{R}^n\\
		u>0 &\text{in} \quad \mathbb{R}^n,\;n\geq 3
	\end{cases}
\end{align}
has a positive solution, if $f\in C^{0,\gamma}(\mathbb{R}^n)$ for  $0<\gamma\leq 1$ and satisfying 
	\begin{align}
		0\leq f(x)\leq L_0/ (1+\vert x\vert^2)^{p/(p-1)}\quad \text{ and } f\not\equiv 0\nonumber
	\end{align}
 whenever  $L_0 =(p-1)\left(\frac{1}{p }\right)^{p/(p-1)}\left[\frac{2}{p-1}\left(n-2-\frac{2}{p-1}\right)\right]^{p/(p-1)}$ for every  $p>n/(n-2)$. However, when  $f(x)>L_0\vert x\vert^{-2p/(p-1)}$ as $\vert x\vert\rightarrow \infty$ he shows nonexistence of positive solutions for every $p>n/(n-2)$.  In the range $1<p\leq n/(n-2)$, he also shows (see \cite[Theorem 5]{Bernard}) that \eqref{Bernard} has no solutions. 
Hence, the exponent $p=n/(n-2)$ and optimal constant $L_0$ are critical numbers for existence or nonexistence of  positive solutions to \eqref{Bernard}. The  critically of  $p=n/(n-2)$ was also observed by Zhang \cite{Zhang} when $0\leq f(x)\leq \frac{\varepsilon}{1+\vert x-x_0\vert^{n+\delta}}$ for some $\delta>0$ and sufficiently small $\varepsilon>0$.  For $p>n/(n-1)$ and bounded smooth domain $\Omega\subset\mathbb{R}^{n+1}_+$, Quittner-Reichel \cite[Theorem 12]{Quittner} shown that 
\begin{align}\label{Q-R}
		-\Delta u=0  \;\;\text{in} \;\;\;\Omega,\quad 
		{\partial_{\nu} u}=u^{p}+f(x) \;\;  \text{on} \;\;\partial\Omega
\end{align}  
has a positive unbounded solution $u\in L^1(\Omega)\times L^1(\partial\Omega)$, for all $f\in L^{\infty}(\partial\Omega)$. The solvability of  \eqref{Q-R} was investigated in $\mathbb{R}^{n+1}_+$ by Ferreira-Medeiros-Montenegro \cite{Ferreira4} when  $\Vert f\Vert_{L^{n(\rho-1)/\rho}(\mathbb{R}^{n})}$ is sufficiently small and $\rho>n/(n-1)$. Let us mention that 
\begin{equation}
L^{n(\rho-1)/\rho}(\mathbb{R}^{n}) \varsubsetneq\mathcal{M}^{n(\rho-1)/\rho}_{p}(\mathbb{R}^{n}) \varsubsetneq \mathcal{M}^{n(\rho-1)/\rho}_{p\infty}(\mathbb{R}^{n}).
\end{equation}
The solvability of  \eqref{Q-R} was also investigated by Ferreira-Montenegro \cite{Ferreira5} when $f$ satisfy the Fourier transform decay $\vert \widehat{f}(\xi)\vert\leq C\varepsilon \vert \xi\vert^{-[(n-1)-1/(\rho-1)]}$, since  \eqref{Q-R} in $\Omega=\mathbb{R}^{n+1}_+$  is equivalent to \eqref{nonlocal1} with  $V\equiv 0$.  These comments leads to \textit{conjecture}: Let $V\equiv 0$, does not exists positive solutions to \eqref{Neumann} or is its equivalent nonlocal version 
\begin{align}
	\begin{cases}
		(-\Delta)^{1/2}u= bu^\rho+f  &\text{in} \quad \mathbb{R}^n\\
		u>0 &\text{in} \quad \mathbb{R}^n,\;n\geq 2
	\end{cases}
\end{align}
when $1<\rho \leq n/(n-1)$ and $0\not\equiv f\in C^{0,\gamma}(\mathbb{R}^n)$ satisfies  $0\leq f(x)\leq K_0 (1+\vert x\vert^2)^{-\rho/2(\rho-1)} $ or when $ f(x)>O(K_0 \vert x\vert ^{-\rho/(\rho-1)})$ for all $\rho>n/(n-1)$. 
This conjecture was partially solved by Zhang-Wang \cite{Zhang-Wang} for nonlocal Lane-Emden equation 
\begin{align}\label{Bernard-nonlocal}
	\begin{cases}
		(-\Delta)^{s} u = u^p +f(x)  &\text{in} \quad \mathbb{R}^n\\
		u>0 &\text{in} \quad \mathbb{R}^n,\;n\geq 3
	\end{cases}
\end{align}
for all $0<s<1$ as $1<p< n/(n-2s)$ and $0\leq f(x)\leq \frac{\varepsilon}{1+\vert x-x_0\vert^{n+\delta}}$ where $\delta>0$ satisfies $n+\delta=(n-2s)(p-1)$ and  $\varepsilon>0$ is sufficiently small. In the case $p=n/(n-2s)$ they shown nonexistence of positive solutions only for $0<s<1/2$. Note that $n+\delta=(n-2s)(p-1)$ implies $1/(1+\vert x\vert^2)^{\frac{sp}{p-1}}\leq C(x_0)\frac{1}{1+\vert x-x_0\vert^{n+\delta}}$ and we can consider $0\leq f(x)\leq \varepsilon/(1+\vert x\vert^2)^{\frac{sp}{p-1}}$  in \cite{Zhang-Wang}. The case $1/2\leq s<1$ for $p=n/(n-2s)$ is still unsolved and optimal constant $K_0$ is still unknown. As we saw above, the size condition for $f$ and range $\rho\in \big(\frac{n}{n-1}, \infty)$  is expected to show existence of positive solutions.  Note that in Theorem  \ref{Symmetries} we can choose nonnegative singular  functions at zero, namely, 
\begin{align*}
f(x) = \varepsilon/\vert x\vert^{\frac{\rho}{\rho -1}}\quad \text{ and } \quad V(x)=c/\vert x\vert 
\end{align*}
for $\varepsilon>0$ and $c>0$ sufficiently small. 


\bigskip 

Our second comment is about locally H\"older regularity of solutions. Let us mention that Hardy-Littlewood-Sobolev Theorem  is known in Morrey spaces (see Olsen \cite{Olsen}), then our fundamental Theorem \ref{linear} (for boundary layers potentials) is true in Morrey spaces.  Hence, mimicking the proof of Theorem \ref{existence} in the Banach space  ${X}^{1}_{rq}(\mathbb{R}^{n+1}_{+}):={X}^{1rq}_{rq}(\mathbb{R}^{n+1}_{+})$ endowed by norm  $$\Vert u \Vert_{{X}^{1}_{rq}(\mathbb{R}^{n+1}_+)}=\Vert\, \nabla u \, \Vert_{\mathcal{M}^{\mu}_{r}(\mathbb{R}^{n+1}_+)}+\Vert \textnormal{tr}_{\partial\mathbb{R}^{n+1}_+}u\Vert_{\mathcal{M}_{q}^{\lambda}(\partial \mathbb{R}^{n+1}_+)}$$ 
and choosing  $f\in \mathcal{M}_p^{n{(\rho-1)}/{\rho}}(\mathbb{R}^{n})$ and  potential $V \in \mathcal{M}_\ell^{n}(\mathbb{R}^{n})$  in Morrey space
$$\mathcal{M}_p^{\lambda}(\mathbb{R}^n)=\Big\{ g\in L^1_{loc}(\mathbb{R}^n)\,:\,  \int_{B_R(x)}\vert g\vert^p dx\leq c\, R^{-n({p}/{\lambda}-1)} \Big\}$$
we will conclude that \eqref{Neumann} has a solution $u\in {X}^{1}_{rq}(\mathbb{R}^{n+1}_{+})$ such that 
\begin{equation}\label{key43}
			\int_{B_R(x,t)\cap \mathbb{R}^{n+1}_+}\vert \nabla u(y,s)\vert^{r}dyds\leq c\prime\, R^{-(n+1)({r}/{\mu}-1)} 
\end{equation}
where the constant $c\prime>0$ only depends of  $C\Vert f\Vert_{\mathcal{M}_p^{\omega}(\mathbb{R}^{n})}$ and $c\Vert V\Vert_{\mathcal{M}_\ell^{n}(\mathbb{R}^{n})}$. The  previous estimate (\ref{key43}) and Poincare's inequality yields 
\begin{equation}
			\int_{B_R(x,t)\cap \mathbb{R}^{n+1}_+}\vert u- u_R\vert^{r}dyds \lesssim R^r\int_{B_R(x,t)\cap \mathbb{R}^{n+1}_+}\vert \nabla u\vert^{r}dyds\lesssim R^{r(1-\frac{n+1}{\mu})+(n+1)} = 
			c\, R^{r\alpha +(n+1)}\nonumber,
\end{equation}
where $\alpha =1-(n+1)/\mu$ and $u_{R}$ denotes the average $u{_{R}}=\Xint{-}_{_{\widetilde{B}_{R}}}u:=\frac{1}{\vert \widetilde{B}_{R}\vert}\int_{_{\widetilde{B}_{R}}}u$ with   $\widetilde{B}_{R}(x,t)={B}_{R}(x,t)\cap\mathbb{R}^{n+1}_+$. Then $u\in \mathcal{E}^{\alpha,r}(\mathbb{R}^{n+1}_+)$, the Campanato space defined by
		\begin{equation}
			\Vert u\Vert_{\mathcal{E}^{\alpha,r}(\mathbb{R}^{n+1}_+)}=\sup_{\widetilde{B}_{R}}R^{-\alpha}\left(\Xint{-}_{\widetilde{B}_{R}}\vert u-u_{R}\vert^r dyds\right)^{1/r}\nonumber.
		\end{equation} 
Hence, from Campanato lemma (see \cite{Campanato, Meyers} for bounded domains and \cite[p.170]{Fabes} for whole space) 
we conclude that $u\in C^{0,\alpha}_{loc}(\overline{\mathbb{R}^{n+1}_+})$ is locally-H\"older continuous with exponent $\alpha =1-(n+1)/\mu$ and $\mu>{n+1}$. 
This is  the issue of the following Corollary.

\begin{corollary}\label{local2} Let  $\rho>n/(n-1)$ such that $n(\rho-1)=\lambda>\mu>n+1$. Under assumptions of  Theorem \ref{existence}, let $V\in\mathcal{M}_\ell^n(\mathbb{R}^n)$ and $f\in\mathcal{M}_p^{\frac{n(\rho-1)}{\rho}}(\mathbb{R}^n)$ for $1<\ell<n$ and $1<p<\frac{n(\rho-1)}{\rho}$.
\begin{itemize}
\item[(A)] There are $\varepsilon>0$ and  $c_k>0$ such that if $\Vert f\Vert_{\mathcal{M}_{p}^{\frac{n(\rho-1)}{\rho}}(\mathbb{R}^n)}\leq \varepsilon/c_1$ and $\Vert V\Vert_{\mathcal{M}_\ell^{n}(\mathbb{R}^{n})}<1/c_2$, the problem \eqref{Neumann} has a solution $u\in X_{rq}^{1}(\mathbb{R}^{n+1}_+)$ such that
\begin{equation}\label{key-1.6}
			\int_{B_R(x,t)\cap \mathbb{R}^{n+1}_+}\vert \nabla u(y,s)\vert^{r}dyds\leq c\prime\, R^{-(n+1)({r}/{\mu}-1)}.\nonumber 
\end{equation}
\item[(B)]  The solution $u$ of the previous item satisfy $u\in C^{0,\alpha}_{loc}(\overline{\mathbb{R}^{n+1}_+})$ with $\alpha =1-(n+1)/\mu$. 
\end{itemize}
\end{corollary}

The paper is organized as follows. In Section \ref{pre} we summarize properties of Lorentz and Morrey-Lorentz spaces. In Section \ref{Riesz-sec2} deal with boundary estimates for double layer  $\mathbf{D}f(x,t)=c\int_{\mathbb{R}^{n}}(|x-y|^2+t^2)^{-\frac{n+1}{2}}tf(y)dy $ and single layer  $\mathbf{S}f(x,t)=c\int_{\mathbb{R}^{n}} (|x-y|^2+t^2)^{-\frac{n-1}{2}}f(y)dy$ in half-space $\mathbb{R}^{n+1}_+$. 
Finally, in Section \ref{Proofs} we prove our theorems.

\section{Preliminaries}\label{pre}
In this section we recall some important properties of Lorentz space $L^{pd}(\Omega)$ and Morrey-Lorentz spaces (see e.g. \cite{BL,Grafakos} for Lorentz spaces and \cite{Ferreira1} for Morrey-Lorentz spaces).
\subsection{Lorentz spaces}
Let $\Omega\subseteq\mathbb{R}^n$ be a measure space endowed with Lebesgue measure $dx$, the Lorentz space $L^{pd}(\Omega)$ is defined to be the set of \text{measurable functions} $f:\Omega\rightarrow\mathbb{R}$ such that
\begin{equation}\label{almost_norm1}
\Vert f\Vert_{L^{pd}(\Omega)}^{\ast}=\left(\frac{d}{p}\int_0^{|\Omega|}[t^{{1}/{p}}f^\ast(t)]^d \frac{dt}{t}\right)^{\frac{1}{d}}<\infty
\end{equation}
with  $1\leq p<\infty$ and $1\leq d<\infty$. For $ 1\leq p\leq  \infty$ with  $d=\infty$, the Lorentz space $L^{p\infty}(\Omega)$ is defined by 
\begin{align}\label{almost_norm}
\Vert f\Vert_{L^{p\infty}(\Omega)}^{\ast}=\sup_{0<t<|\Omega|} t^{1/p}f^\ast(t),
\end{align}
where $f^{\ast}(t)$ denotes the decreasing rearrangement $
\text{ }f^{\ast}(t)=\inf\{s>0:d_{f}(s)\leq t\}$ for  $ d_{f}(s):=\mu(\{x\in\Omega:|f(x)|>s\})
$.
The Lorentz space $L^{pd}(\Omega)$ increase with index $d$ in the sense of  continuous inclusions
\begin{equation}
L^{p1}(\Omega)\subset L^{pd_{1}}(\Omega)\subset L^{p}(\Omega)\subset L^{pd_{2}}(\Omega)\subset
L^{p\infty}(\Omega) \label{inclusion-1}
\end{equation}
 for all  $1<d_{1}\leq p\leq d_{2}<\infty.$
Consider the interpolation functor
$(\cdot,\cdot)_{\theta,d}$ which can be constructed via the $K_{\theta,d}-$method (see e.g. \cite{BL}). Let 
$0<p_{1}<p<p_{2}\leq\infty$ and $\theta\in (0,1)$ be such that $\frac{1}{p}=\frac{1-\theta}{p_{1}}+\frac{\theta
}{p_{2}}$, then we get  the real interpolation property (see \cite{Hunt} or \cite[Theorems 5.3.1
and 5.3.2]{BL})
$$\left(  L^{p_{1}d_{1}},L^{p_{2}d_{2}}\right)  _{\theta,d}=L^{pd}$$
for all $1\leq d_{1},d,d_{2}\leq\infty.$
Also, the multiplication operator $T_f(g)=fg\,$ works well in Lorentz spaces.
\begin{lemma}[Theorem 3.4 and 3.5 at \cite{O'Neil}]\label{HOLDER} Let $1\leq p_{1},p_{2}\leq\infty$, $1<r\leq \infty$ and $1\leq z_{1},z_{2}\leq \infty$ be such that
	\begin{equation}
\frac{1}{r}=\frac{1}{p_1}+\frac{1}{p_2}\;\; \text{ and }\;\; \frac{1}{z_1}+\frac{1}{z_2}\geq\frac{1}{s},\nonumber
	\end{equation}
where $s\geq 1$. If $f\in L^{p_1z_1}$ and $g\in L^{p_2z_2}$, then
	\begin{equation} \label{holder}
	\Vert fg\Vert_{rs}\leq \frac{r}{r-1}\Vert f\Vert_{p_{1}z_{1}}\Vert
	g\Vert_{p_{2}z_{2}}.
	\end{equation}
If $f\in L^{p_1z_1}$ and $g\in L^{p_1'z_2}$, then
 	\begin{equation} \label{holder2}
 \Vert fg\Vert_{L^1}\leq \Vert f\Vert_{p_{1}z_{1}}\Vert
 g\Vert_{p_{1}'z_{2}}.\nonumber
 \end{equation}
\end{lemma}
An extension of Minkowski inequality for integrals can be easily obtained from  duality (see \cite{Grafakos}) between $L^{pd}$ and $L^{p'd'}$, where $p'$ and $d'$ are exponential conjugate of $p$ and $d$, respectively. 

\begin{lemma}[Minkowski inequality in $L^{pd}$-spaces \cite{Ferreira1}]  Let $K(\cdot,y)\in L^{rd}(\mathbb{R}^n)$ for  each $y\in\mathbb{R}^n$ and let $f\in L^1_{loc}(\mathbb{R}^n)$, then
\begin{equation}\label{Minkowski}
\left\Vert \int_{\mathbb{R}^n}K(x,y)f(y)dy\right\Vert_{L^{pd}(\mathbb{R}^n,dx)}\leq  \int_{\mathbb{R}^n}\left\Vert K(x,y)\right\Vert_{L^{pd}(\mathbb{R}^n,dx)}\vert f(y)\vert dy,
\end{equation}
for every $1< p\leq \infty $ and $1\leq d\leq \infty$.
\end{lemma}

\subsection{Morrey-Lorentz spaces}
Let  $B_{\ell}(x_0)$ be a ball in $\mathbb{R}^n$ centered in $x_0$ with radius $\ell>0$ and consider $Q_{\ell}=B_\ell(x_0)\cap\Omega$. Let  $1\leq r\leq p<\infty$, $1\leq s\leq\infty$, we say that a function $f\in L^{rs}(Q_\ell)$ belongs to Morrey-Lorentz space $\mathcal{M}^\mu_{rs}(\Omega)$ if 
\begin{equation}
	\Vert f \Vert_{\mathcal{M}^\mu_{rs}}=\sup_{Q_{\ell}}\vert Q_\ell\vert^{\frac{1}{\mu}-\frac{1}{r}}\Vert f\Vert_{L^{qs}(Q_{\ell})}\nonumber
\end{equation}
is finite, where the supremum is taken over ball $Q_\ell$. The space  $\mathcal{M}^\mu_{r\infty}(\Omega)$ denotes the well-known weak-Morrey spaces, namely, weak-$\mathcal{M}^\mu_{r}(\Omega)$ and $\mathcal{M}^\mu_{r,r}(\Omega)$ denotes the so-called Morrey space $\mathcal{M}^\mu_{r}(\Omega)$ defined by 
\begin{equation}
	\Vert f \Vert_{\mathcal{M}^\mu_{r}(\Omega)}=\sup_{Q_{\ell}}\vert Q_\ell\vert^{\frac{1}{\mu}-\frac{1}{r}}\Vert f\Vert_{L^{r}(Q_{\ell})}\nonumber.
\end{equation}
The Morrey-Lorentz spaces is also denoted by  $\mathcal{M}_{r,s,\lambda}$ with $\lambda=n(1-r/\mu)$ and $0\leq \lambda<n$. 

\begin{lemma}[H\"older's inequality \cite{Ferreira1}] Let $1<q_j\leq \lambda_j<\infty$ and $1<r\leq \beta<\infty$ such that $\small \frac{1}{r}=\frac{1}{q_1}+\frac{1}{q_2}$ and  $\frac{1}{\beta}=\frac{1}{\lambda_1}+\frac{1}{\lambda_2}$, then
\begin{equation}
\Vert fg\Vert_{\mathcal{M}_{r,s}^{\beta}(\mathbb{R}^n)}\leq C \Vert f\Vert_{\mathcal{M}_{q_1,d_1}^{\lambda_1}(\mathbb{R}^n)} \Vert g\Vert_{\mathcal{M}_{q_2,d_2}^{\lambda_2}(\mathbb{R}^n)},\label{weak-holder}
\end{equation}	
for all $s\geq1$ satisfying $\frac{1}{d_1}+\frac{1}{d_2}\geq \frac{1}{s}$, where $C>0$ is a universal constant. 
\end{lemma}

\section{Potential estimates}\label{Riesz-sec2}
According to \cite[Remark 1.2]{Marcelo-Lidiane} in Morrey-Lorentz spaces  and  \cite{Olsen} in Morrey spaces we have the following theorem.
\begin{theorem}[Hardy-Littlewood-Sobolev Theorem \cite{Marcelo-Lidiane},\cite{Olsen}] \label{HLS} Let $1<p\leq \lambda<\infty$ and $1<r\leq\mu <\infty$ be such that  $r/\mu \leq p/\lambda$ and $1<p<r<\infty$. If $f\in \mathcal{M}_{pd}^{\lambda}(\mathbb{R}^n)$,  there exits a constant $C>0$ such that 
$$\Vert I_{\alpha}f\Vert_{\mathcal{M}_{rs}^{\mu}(\mathbb{R}^n)}\leq C \Vert f\Vert_{\mathcal{M}_{pd}^{\lambda}(\mathbb{R}^n)}$$
if provided $\alpha+n/\mu={n}/{\lambda}$, $0< \alpha <{n}/{\lambda}$ and $1\leq d<s \leq \infty$, where $I_{\alpha}$ is defined  by 
\begin{align*}
	I_{\alpha}f(x)=C_{\alpha,n}\int_{\mathbb{R}^n}\vert x-y\vert ^{\alpha-n}f(y)dy\quad \text{ a.e. } x\in\mathbb{R}^n \quad \text{ as }\quad 0<\alpha<n.
\end{align*}
\end{theorem}

\begin{remark}[Optimality]\label{T-sharp}The Optimality of  ${r}/{\mu}\leq {p}/{\lambda}$ was proved in \cite{Olsen} for Morrey spaces. 	 
\end{remark}

\subsection{Boundary estimates}
Consider the double layer potential $z=\mathbf{D}f$ (respectively $z=\mathbf{S}f$) on half-space $\mathbb{R}^{n+1}_+$,
\begin{equation}
\mathbf{D}f(x,t)=\int_{\mathbb{R}^{n}}\partial_{\nu}G(x-y,t)f(y)dy 
\end{equation}
which is the unique solution, modulus constants, to the problem
\begin{align*}
\begin{cases}
\Delta z=0 & \text{in} \;\;\;\mathbb{R}^{n+1}_{+},\\
z|_{t=0} = f(x) \; & \text{on} \;\;\partial \mathbb{R}^{n+1}_+ \quad (\text{respectively } (\partial_{\nu}z)|_{t=0}=f(x)).
\end{cases}
\end{align*}
As we known the Poisson kernel $P_{t}(y)=\partial_{\nu}G(y,t)$ is given by $\frac{\Gamma((n+1)/2)}{\pi^{(n+1)/2}}\frac{t}{(\vert y\vert ^2+t^2)^{(n+1)/2}}$ for $n\geq 2$. Our first result is this section is the following.
\begin{theorem}\label{linear}  Let $n\geq 2$, $1<p\leq \omega<\infty$ and $1<r\leq\mu<\infty$. If $1<p<r<\infty$  and $\omega/n<r\leq s$ satisfies 
\begin{align*}
\frac{r}{\mu}\leq\frac{p}{\omega} \quad \text{ and }\quad \frac{n}{\omega}=\frac{n}{\mu}+\frac{1}{r}
\end{align*}then, there exists  a constant $C>0$ (independent of $f$) such that 
	\begin{equation}
\Vert \mathbf{D}f\Vert_{\mathcal{M}_{rs}^{\mu}(\mathbb{R}^{n+1}_+)} \leq C \Vert f\Vert_{\mathcal{M}_{pd}^{\omega}(\mathbb{R}^{n})} \label{Pot-bound1}
		\end{equation}
for all $1\leq d< s\leq \infty$. In addition, if   $\frac{n}{\omega}=\frac{n}{\mu}+1$ then there exists  a constant $C>0$ (independent of $f$) such that
		\begin{equation}\label{Pot-bound2}
		\Vert \normalfont\text{tr}_{\partial\mathbb{R}^{n+1}_+}(\mathbf{S}f)\Vert_{\mathcal{M}_{rs}^{\mu}(\mathbb{R}^{n}) } \leq C \Vert f\Vert_{\mathcal{M}_{pd}^{\omega}(\mathbb{R}^{n})},
		\end{equation}
for all $1\leq d< s\leq \infty$
\end{theorem}

\begin{proof} 
Fix $x_0=(x'_0\,,\, t_0)\in\mathbb{R}^{n+1}_+$ and let $B_{\delta}\subseteq \mathbb{R}^{n+1}_+$ a ball with center $(x'_0\,,\, t_0)$ and radius $\delta=2t_0$. Clearly $B_\delta\subset Q_{\delta}\times I_{\delta}$, where $Q_\delta=B_\delta'$ denotes the projection of  $B_\delta$ onto  $\partial\mathbb{R}^{n+1}_+$ and $I_{\delta}=[\delta/4,3\delta/4]$.  It follows from harmonicity of  $z=\mathbf{D}f$ that 
	\begin{align*}
\vert \mathbf{D}f(x_0)\vert &\leq \frac{1}{\vert B_\delta\vert}\int_{B_\delta}\vert \mathbf{D}f\vert dx\\
&\leq \frac{1}{\vert B_\delta\vert}\int_{Q_\delta}\int_{\delta/4}^{3\delta/4}\vert \mathbf{D}f(x',t)\vert dt\,dx'
	\end{align*}
which from Fubini's theorem yields
	\begin{align*}
\int_{B_\delta}\vert \mathbf{D}f\vert dx &\leq \int_{Q_\delta}\int_{\delta/4}^{3\delta/4}\vert \mathbf{D}f(x',t)\vert dt\,dx'\\
&=\vert Q_\delta \times I_\delta\vert ^{1-\frac{1}{r_j}}\left(\vert Q_\delta \times I_\delta\vert ^{\frac{1}{r_j}-1} \int_{Q_\delta}\int_{I_{\delta}}\vert \mathbf{D}f(x',t)\vert dt dx'\right).
	\end{align*}
Since $\vert B_{\delta}\vert=c\vert Q_\delta\vert\, \vert I_{\delta}\vert $ we have 
	\begin{align}
\frac{1}{\vert B_\delta\vert^{1-1/r_j}}\int_{B_\delta}\vert \mathbf{D}f\vert dx 
&\leq  C_j\; \vert Q_\delta \times I_\delta\vert ^{\frac{1}{r_j}-1} \int_{Q_\delta}\int_{I_{\delta}}\vert \mathbf{D}f(x',t)\vert dt dx'\label{key-esp-p2}.
	\end{align}
Now recall that every $f\in L^{r\infty}(X,d\mu)$ with $1<r<\infty$ one has 
$$ \Vert f\Vert_{L^{r\infty}(X,d\mu)}\sim \sup\left \{\mu(E)^{\frac{1}{r}-1}\int_{E}\left\vert f(y)\right\vert d\mu \,:\, \text{ measurable } E\subseteq X \text{ and } 0<\mu(E)<\infty\right\},$$ 
which from \eqref{key-esp-p2} yields
	\begin{equation}
	\Vert \mathbf{D}f\Vert_{L^{r_j\infty}(B_\delta)} \leq C_{j}\Vert  \mathbf{D}f\Vert_{L^{r_j\infty}_{_{x\prime,\, t}}\left(Q_\delta \times I_{\delta}\right)},\nonumber \quad  j=1,2
	\end{equation}
where  $\Vert u(x',t)\Vert_{L_{_{x\prime,\,t}}^{rs}}$ is defined by 
$$\Vert u\Vert_{L_{_{x\prime,\,t}}^{rs}(Q_{\delta}\times I_{\delta})}:=\left\Vert \left\Vert u\right\Vert_{L^{rs}(I_{\delta})}\right\Vert_{L^{rs}(Q_{\delta})}.$$
Let $1<r_1<r<r_2<\infty$. From real interpolation properties (see e.g. \cite[Theorem 5.3.1]{BL})
	\begin{equation*}
	\left( L^{r_{1}\infty},L^{r_{2}\infty}\right) _{\theta,s}=L^{rs}\quad \text{ and } \quad 
	\left(L_{_{x\prime,\,t}}^{r_1\infty},\,L_{_{x\prime,\,t}}^{r_2\infty}\right)_{\theta,s}=L_{_{x\prime,\,t}}^{rs},
	\end{equation*}
with  $\theta\in (0,1)$ such that  $\frac{1}{r}=\frac{1-\theta}{r_1}+\frac{\theta}{r_2}$,  we conclude via Marcinkiewicz interpolation theorem \cite[Theorem 5.3.2]{BL} that
\begin{align}\label{key-ineq-interpoled}
	\Vert \mathbf{D}f\Vert_{L^{rs}(B_\delta)}\leq C_1^{1-\theta}C_2^\theta \left\Vert \left\Vert \mathbf{D}f\right\Vert_{L^{rs}(I_{\delta})}\right\Vert_{L^{rs}(Q_{\delta})}.
	\end{align}
Since $L^r(0,\infty)\hookrightarrow L^{rs}(0,\infty)$ for  $1<r\leq s$, it follows from Minkowski inequality (\ref{Minkowski}) that 
	\begin{align}
	\Vert (\mathbf{D}f)(x',\cdot)\Vert_{L^{rs}(I_{\delta})}&\leq  \int_{\mathbb{R}^{n}}\Vert \partial_t G(y,t)\Vert_{L^{rs}(I_{\delta})} |f(x'-y)|dy\nonumber\\
&\leq C \int_{\mathbb{R}^{n}}\left(\int_{0}^{\infty} \frac{dt }{(\vert y\vert^2+t^2)^{nr/2}}dt\right)^{1/r} |f(x'-y)|dy\nonumber\\
	&\leq C \int_{\mathbb{R}^{n-1}}\frac{1}{|y|^{n-1/r}}|f(x'-y)|dy\label{key2},
	\end{align}
since $\int_{0}^{\infty} (\vert y\vert^2+t^2)^{-nr/2}dt = \frac{\pi}{2}\frac{1}{\vert y\vert^{nr\,-\,1}}$.  Inserting \eqref{key2} into right-hand side of \eqref{key-ineq-interpoled} and invoking Theorem \ref{HLS} with   $\alpha=\frac{1}{r}$ we have
	\begin{align*}
	\vert B_\delta\vert^{\frac{1}{\mu}-\frac{1}{r}}\Vert \mathbf{D}f\Vert_{L^{rs}(B_\delta)}&\leq C \vert Q_\delta \vert^{\frac{1}{\mu}-\frac{1}{r}}\left\Vert\int_{\mathbb{R}^{n}}\frac{1}{|x'-y|^{n-1/r}}|f(y)|dy\right\Vert_{L^{rs}(Q_\delta)}\\
	&\leq C \left\Vert\int_{\mathbb{R}^{n}}\frac{1}{|x'-y|^{n-1/r}}|f(y)|dy\right\Vert_{\mathcal{M}_{rs}^{\mu}(\mathbb{R}^{n})}\\
	&\leq C \Vert f\Vert_{\mathcal{M}_{pd}^{\omega}(\mathbb{R}^{n})},
	\end{align*}
whener $\frac{1}{r}+\frac{n}{\mu}=\frac{n}{\omega}\,$ with $\alpha=1/r<n/\omega\,$ and $\,{r}/{\mu}\leq {p}/{\omega}\,$  
as we wish to show.  Since $\text{tr}_{\partial\mathbb{R}^{n+1}_+}(\mathbf{S}f)(x',t)=I_{1}f(x')\,$  almost everywhere $x'\in \mathbb{R}^{n}$, the estimate \eqref{Pot-bound2} is a direct consequence of Theorem \ref{HLS}.
\end{proof}

Let $f(x,t)$ be a function defined on half-space $\mathbb{R}^{n+1}_+$ and let $\hat{f}(\xi, t)$ be its Fourier transform in $x$ defined by 
\begin{equation}
\hat{f}(\xi, t)= (2\pi)^{-n/2}\int_{\mathbb{R}^{n}}e^{-i x'\cdot \xi}f(x,t)dx\nonumber.
\end{equation}
Note that Dirichlet operator has the symbol $\sigma(\mathbf{D})=\exp\{-t\vert\xi\vert\}$ and Neumann operator has the symbol $\sigma(\mathbf{S})=-\exp\{-t\vert\xi\vert\}/\vert\xi\vert $, then 
\begin{align*}
(\partial_j\mathbf{S}f)^{\wedge}(\xi,t)=-\exp\{-t\vert\xi\vert\}\,\frac{i\xi_j}{\vert\xi\vert}\widehat{f}(\xi,t)=(\mathbf{D}S_jf)^{\wedge}(\xi,t ),
\end{align*}
where $S_j$ denotes the tangential $j$-th Riesz transform acting on $\mathbb{R}^{n+1}_+$ or $\mathbb{R}^{n}$ whose symbol is given by $\sigma(S_j)=-\frac{i\xi_j}{\vert\xi\vert}$ for $j=1,\cdots,n$. Hence, we have the following lemma.
\begin{lemma}\label{prop-D-Riesz}Let $n\geq 2$ and $f\in\mathcal{S}(\mathbb{R}^{n})$, then
	\begin{itemize}
	\item[(i)] $\partial_{\nu} \mathbf{S}f=\mathbf{D}f$
	\item[(ii)] $ \partial_j\mathbf{S}f=\mathbf{D}S_jf=S_j\mathbf{D}f$,\; \text{ for } \; $j=1,2,\cdots,n$.
	\end{itemize}
\end{lemma}
\noindent According to Calderon-Zygmund theory, the singular integral operator
\begin{equation}
(R_jf)(x)=\,\textbf{P.V.} \;C_n\int_{\mathbb{R}^n} \frac{x_j-y_j}{|x-y|^{n+1}}f(y)dy\label{Riesz-transf}
\end{equation}
is bounded from $L^p(\mathbb{R}^n)$ to itself, for $1<p<\infty$. Hence, splinting the kernel $K(x)=c_n\frac{x_j}{\vert x\vert^{n+1}}$ into
\[ K(x)=K(x)\chi_{\{ |x|\leq r\}}+K(x)\chi_{\{ |x|>r\}}\]
one can show that (see  \cite[Lemma 2.3]{Ferreira1}) that  $R_j$ is bounded from  $\mathcal{M}_{pd}^{\omega}(\mathbb{R}^n)$ to itself, for every $1<p\leq \omega<\infty$ and $1<d\leq\infty$. It follows that $S_j$ is a continuous singular operator on  $\mathcal{M}_{pd}^{\omega}(\mathbb{R}^{n})$ which implies directly from Lemma \ref{prop-D-Riesz} and Theorem \ref{linear} the following bound for gradiente $\nabla_{x,t}\mathbf{S}f$ in Morrey-Lorentz spaces. 

\begin{corollary} \label{linear-main} 
Let $n\geq 2$, $1<p\leq \omega<\infty$ and $1<r\leq\mu<\infty$. If $1<p<r<\infty$ and $\omega/n<r\leq s$ satisfies ${r}/{\mu}\leq{p}/{\omega}$ and ${n}/{\omega}={n}/{\mu}+{1}/{r}$, there exists a constant $C>0$ such that 
	\begin{equation}\label{Pot-bound3}
	\left\Vert  \partial_{ x_j}\mathbf{S}f\right\Vert_{\mathcal{M}_{rs}^{\mu}(\mathbb{R}^{n+1}_+)} \leq C \Vert f\Vert_{\mathcal{M}_{pd}^{\omega}(\mathbb{R}^{n})}\quad j=1,2,\cdots,n+1\nonumber
	\end{equation}
	for all $f\in \mathcal{M}_{pd}^{\omega}(\mathbb{R}^{n})$, where $1<d<s\leq \infty$. Here $\partial_{t}:=\partial_{x_{n+1}}$.
\end{corollary}

\bigskip

\subsection{Linear problem}
For solvability of  \eqref{Linear} let us rewritten this problem as an integral equation 
\begin{align}\nonumber
	u(x,t)&= c_n \int_{\partial\mathbb{R}^{n+1}_+}(\vert x-y\vert^2+t^2)^{\frac{n-1}{2}}(V(y)u(y,t)+f(y))dy:= \mathbf{S}f+\mathcal{T}_{V}(u).
\end{align} 
In the next Lemma we assume that $1<r<\mu<\infty$, $1<q<\lambda<\infty$ and $1<p<\omega<n<nr$ satisfies 
	\begin{equation}
	\frac{n}{\omega}=\frac{n}{\mu}+\frac{1}{r}\quad \text{ and }\quad \frac{n}{\omega}=\frac{n}{\lambda}+1,\label{ineq-scaling}
	\end{equation}
for all $\{r/\mu,q/\lambda\}\leq p/\omega$ and $n\geq 2$. These assumptions implies that $\frac{n}{\mu} -\frac{n}{\lambda}=\frac{1}{r'}$ and $\lambda=n\omega/(n-\omega)$. 

\begin{lemma}\label{linear-thm} Let  $f\in\mathcal{M}_{p\infty}^{\omega}(\mathbb{R}^{n})$ and let $V\in \mathcal{M}_{\ell\infty}^{n}(\mathbb{R}^{n})$ such that $\Vert V\Vert_{\mathcal{M}_{\ell \infty}^{n}(\mathbb{R}^{n})}<1/c$ for some $c>0$, then (\ref{Linear}) has a unique solution $u\in {X}^{1,\infty,\infty}_{rq}(\mathbb{R}^{n+1}_{+})$ such that 
	\begin{equation}
		\Vert \nabla u \Vert_{\mathcal{M}^{\mu}_{r\infty}(\mathbb{R}^{n+1}_+)}\leq c\prime\Vert f\Vert_{\mathcal{M}_{p\infty}^{\omega}(\mathbb{R}^{n})} \quad \text{ and }\quad \Vert \textnormal{tr}_{\partial\mathbb{R}^{n+1}_+}u\Vert_{\mathcal{M}_{q\infty}^{\lambda}(\partial \mathbb{R}^{n+1}_+)}\leq c\prime\Vert f\Vert_{\mathcal{M}_{p\infty}^{\omega}(\mathbb{R}^{n})}\nonumber,
	\end{equation}
provided that $\frac{n}{\mu} -\frac{n-\omega}{\omega}=\frac{1}{r'}$ as $\lambda=\frac{n\omega}{n-\omega}$ and $1< \ell< n$, where $c\prime= C\Vert f\Vert_{\mathcal{M}^{\omega}_{p\infty}}/\big(1-c\Vert V\Vert_{\mathcal{M}_{\ell \infty}^{n}}\big)$.
\end{lemma}

\begin{proof}Consider the map $\Phi: u\mapsto \Phi(u):= \mathbf{S}f+\mathcal{T}_{V}(u)$ for $u\in B_{\varepsilon}$ which is a closed ball in ${X}^{1,\infty,\infty}_{rq}(\mathbb{R}^{n+1}_+)$ defined by 
 \begin{equation}
 B_{\varepsilon}=\left\{ u\in {X}^{1,\infty,\infty}_{rq}(\mathbb{R}^{n+1}_+): \Vert u\Vert_{X^{1,\infty,\infty}_{rq}(\mathbb{R}^{n+1}_+)}\leq \frac{2\varepsilon}{1-L}\right\}  \nonumber,
 \end{equation}
where $\varepsilon=C \Vert f\Vert_{\mathcal{M}_{p\infty}^{\omega}}$ and $L= c\Vert V\Vert_{\mathcal{M}_{\ell \infty}^{n}(\mathbb{R}^{n})}<1$. We will show that $\Phi$ has a fixed point in the Banach space ${X}^{1,\infty,\infty}_{rq}(\mathbb{R}^{n+1}_+)$. 

Based on the assumptions \eqref{ineq-scaling} we can invoke Corollary \ref{linear-main} and estimate \eqref{Pot-bound2} to get the control 
\begin{equation}
\Vert \mathbf{S}(f)\Vert_{{X}^{1,\infty,\infty}_{rq}(\mathbb{R}^{n+1}_+)}= \Vert \nabla\, \mathbf{S}f\Vert_{\mathcal{M}^\mu_{r\infty}(\mathbb{R}^{n+1}_+)}+ \Vert  \textnormal{tr}_{\partial\mathbb{R}^{n+1}_+}(\mathbf{S}f)\Vert_{\mathcal{M}_{q\infty}^{\lambda}(\partial\mathbb{R}^{n+1}_+)}\leq C \Vert f\Vert_{\mathcal{M}_{p\infty}^{\omega}(\partial\mathbb{R}^{n+1}_+)},\label{L-est}
\end{equation}
for all $f\in \mathcal{M}_{p\infty}^{\omega}(\partial\mathbb{R}^{n+1}_+)$.  Let $1<p_0<\omega_0<n$ such that (\ref{ineq-scaling}) holds for $\omega_0=\omega$ and $p_0=p$. Moreover, consider 
\begin{equation}
\frac{1}{\omega_0}=\frac{1}{n}+\frac{1}{\lambda}\;\; \text{ and \;\;} \frac{1}{p_0}=\frac{1}{\ell}+\frac{1}{q}.\nonumber
\end{equation}
The previous estimate \eqref{L-est} and H\"older's inequality \eqref{weak-holder} give us
	\begin{align}
	\Vert \mathcal{T}_V(u)\Vert_{{X}^{1,\infty,\infty}_{rq}(\mathbb{R}^{n+1}_+)}= \Vert \mathbf{S}(Vu)\Vert_{{X}^{1,\infty,\infty}_{rq}(\mathbb{R}^{n+1}_+)}&\leq  c\Vert V\Vert_{\mathcal{M}_{\ell\infty}^{n}(\partial\mathbb{R}^{n+1}_+)} \Vert u\vert_0\Vert_{\mathcal{M}_{q\infty}^{\lambda}(\partial\mathbb{R}^{n+1}_+)}\nonumber\\
&\leq  c\Vert V\Vert_{\mathcal{M}_{\ell\infty}^{n}(\mathbb{R}^n)} \Vert u\Vert_{{X}^{1,\infty,\infty}_{rq}(\mathbb{R}^{n+1}_+)}\label{lin-contrac}.
	\end{align}
It follows that $\Phi(u)\in B_{\varepsilon}$ for  $u\in B_{\varepsilon}$, since 
	\begin{align}
	\Vert \Phi(u)\Vert_{{X}^{1,\infty,\infty}_{rq}(\mathbb{R}^{n+1}_+)}\leq \Vert \mathbf{S}f\Vert_{{X}^{1,\infty,\infty}_{rq}(\mathbb{R}^{n+1}_+)}+\Vert \mathcal{T}_{V}(u)\Vert_{{X}^{1,\infty,\infty}_{rq}(\mathbb{R}^{n+1}_+)} &\leq C \Vert f\Vert_{\mathcal{M}_{p\infty}^{\omega}} + L\Vert u\Vert_{{X}^{1,\infty,\infty}_{rq}(\mathbb{R}^{n+1}_+)}\nonumber\\
&\leq \varepsilon +\frac{2\varepsilon L}{1-L}\nonumber\\
&<\frac{2\varepsilon}{1-L}\nonumber,
	\end{align}
for all $0<L<1$. Moreover, we have
		\begin{equation}
	\Vert \Phi(u)-\Phi(v)\Vert_{{X}^{1,\infty,\infty}_{rq}(\mathbb{R}^{n+1}_+)}=\Vert \mathcal{T}_V(u-v)\Vert_{{X}^{1,\infty,\infty}_{rq}(\mathbb{R}^{n+1}_+)}\leq L\Vert u-v\Vert_{{X}^{1,\infty,\infty}_{rq}(\mathbb{R}^{n+1}_+)}\nonumber.
	\end{equation}
The previous estimates show us that $\Phi$ is a contraction, then it has exactly a fixed point $u\in B_{\varepsilon}$ endowed by distance  $d(u,v)=\Vert u-v\Vert_{{X}^{1,\infty,\infty}_{rq}(\mathbb{R}^{n+1}_+)}$, as we wish to show.
\end{proof}

\section{Proof of  main theorems}\label{Proofs}

\noindent\textbf{Proof of Theorem \ref{existence}.} Let us rewrite \eqref{Neumann} as an integral equation
\begin{align}\label{key-eq-solve}
	u(x,t)=\mathbf{S}(b\vert u\vert^{\rho-1} u)+ \mathbf{S}(Vu)+\mathbf{S}f=\mathcal{B}(u)+\mathcal{T}_{V}(u)+ \mathbf{S}f
\end{align} 
and recall of the inequality
\[
\left\vert |u|^{\rho-1}u-|v|^{\rho-1}v\right\vert \leq C|u-v|(|u|^{\rho-1}+|v|^{\rho-1}) \text{ for }\rho\geq 1.
\]
Now writing $\frac{\rho}{q}=\frac{1}{q}+\frac{\rho-1}{q}$ and  $\frac{1}{\lambda/\rho}= \frac{1}{\lambda}+\frac{1}{\lambda/(\rho-1)}$, we get  from H\"older's inequality \eqref{weak-holder} 
\begin{align}
\left\Vert \,\left\vert |u|^{\rho-1}u-|v|^{\rho-1}v\right\vert \,\right\Vert_{\mathcal{M}_{(q/\rho) \infty}^{\lambda/\rho}}&\leq C \left\Vert u-v \right\Vert_{\mathcal{M}_{q, \infty}^{\lambda}} \left \Vert |u|^{\rho-1}+|v|^{\rho-1}\right\Vert_{\mathcal{M}_{{q}/{(\rho-1)}, \infty}^{\lambda/(\rho-1)}}\notag\\
&\leq C \left\Vert u-v \right\Vert_{\mathcal{M}_{q,\infty}^{\lambda}} \left (  \Vert u\Vert_{\mathcal{M}_{q, \infty}^{\lambda}}^{\rho-1}+\Vert v\Vert_{\mathcal{M}_{q, \infty}^{\lambda}}^{\rho-1}\right)\notag
\end{align}
which from Corollary \ref{linear-main} and estimate \eqref{Pot-bound2} yields 
\begin{align}
	\Vert \mathcal{B}(u)-\mathcal{B}(v) \Vert_{{X}^{1,\infty,\infty}_{rq}(\mathbb{R}^{n+1}_+)}&\leq  b\left\Vert \nabla \mathbf{S}\left( |u|^{\rho-1}u-|v|^{\rho-1}v\right)\right\Vert_{\mathcal{M}_{r\infty}^{\mu}}+ b\left\Vert \textnormal{tr}_{\partial\mathbb{R}^{n+1}_+} \mathbf{S} \left( |u|^{\rho-1}u-|v|^{\rho-1}v\right)\right\Vert_{\mathcal{M}_{q\infty}^{\lambda}}\notag\\
	&\leq C \left\Vert |u|^{\rho-1}u-|v|^{\rho-1}v\right\Vert_{\mathcal{M}_{(q/\rho)\infty}^{\lambda/\rho}(\partial\mathbb{R}^{n+1}_+)} \nonumber\\
	&\leq C \left\Vert u-v \right\Vert_{\mathcal{M}_{q\infty}^{\lambda}(\partial\mathbb{R}^{n+1}_+)} \left (  \Vert u\Vert_{\mathcal{M}_{q\infty}^{\lambda}(\partial\mathbb{R}^{n+1}_+)}^{\rho-1}+\Vert v\Vert_{\mathcal{M}_{q \infty}^{\lambda}(\partial\mathbb{R}^{n+1}_+)}^{\rho-1}\right)\label{key1-wellp},
\end{align}
if provided that $1<r<\mu<\infty$ and $1<(q/\rho)<(\lambda/\rho) <\infty$ satisfies $r/\mu\leq q/\lambda= (q/\rho)/(\lambda/\rho)$ and 
\begin{equation}\label{exactly}
\frac{n}{\lambda/\rho}=1+\frac{n}{\lambda}\quad \text{ and }\quad \frac{n}{\lambda/\rho}=\frac{n}{\mu}+\frac{1}{r}.
\end{equation}
The last conditions implies $\lambda=n(\rho-1)$,  $\frac{\rho}{\rho-1}=\frac{n}{\mu}+\frac{1}{r}$ and $\frac{n}{\mu}=\frac{n}{\lambda}+\frac{1}{r'}$, where $r'$ denotes the  conjugate exponent of $r$. 

Now we will show that  $u\mapsto \Psi(u):= \mathbf{S}f+\mathcal{T}_{V}(u)+\mathcal{B}(u)$ for $u\in B_{\varepsilon}=\{ u\in {X}^{1,\infty,\infty}_{rq}(\mathbb{R}^{n+1}_+)\, :\, \Vert u\Vert_{{X}^{1,\infty,\infty}_{rq}(\mathbb{R}^{n+1}_+)}\leq \frac{2\varepsilon}{1-L}\}$ is a contraction in  ${X}^{1,\infty,\infty}_{rq}(\mathbb{R}^{n+1}_+)$ such that $\Psi(B_{\varepsilon})\subseteq B_{\varepsilon}$, where $L=c\Vert V\Vert_{\mathcal{M}_{\ell, \infty}^{n}(\mathbb{R}^n)}<1$ and $\Vert f\Vert_{\mathcal{M}_{p\infty}^{\omega}(\mathbb{R}^{n})}\leq \varepsilon/C$  for $\varepsilon>0$ sufficiently small. 
Indeed, from estimates \eqref{lin-contrac} and \eqref{key1-wellp} one has
	\begin{align}
	\Vert \Psi(u)-\Psi(v)\Vert_{{X}^{1,\infty,\infty}_{rq}(\mathbb{R}^{n+1}_+)}&\leq \Vert \mathcal{B}(u)-\mathcal{B}(v)\Vert_{{X}^{1,\infty,\infty}_{rq}(\mathbb{R}^{n+1}_+)} +  \Vert \mathcal{T}_{V}(u)-\mathcal{T}_{V}(v)\Vert_{{X}^{1,\infty,\infty}_{rq}(\mathbb{R}^{n+1}_+)}\notag\\
	&\leq C \left\Vert u-v \right\Vert_{{X}^{1,\infty,\infty}_{rq}(\mathbb{R}^{n+1}_+)} \big(\Vert u\Vert_{{X}^{1,\infty,\infty}_{rq}(\mathbb{R}^{n+1}_+)}^{\rho-1}+\Vert v\Vert_{X}^{\rho-1}\big) +L\Vert u-v\Vert_{{X}^{1,\infty,\infty}_{rq}(\mathbb{R}^{n+1}_+)} \notag\\
	&\leq \Vert u-v\Vert_{{X}^{1,\infty,\infty}_{rq}(\mathbb{R}^{n+1}_+)}  \Big( C \Big(\frac{2\varepsilon}{1-L}\Big)^{\rho-1}+C \Big(\frac{2\varepsilon}{1-L}\Big)^{\rho-1}+L\Big)\notag\\
	&\leq \Big( 2C\Big(\frac{2\varepsilon}{1-L}\Big)^{\rho-1} +L\Big) \Vert u-v\Vert_{X} .\label{key-est-exist}
	\end{align}
Note that \eqref{exactly} is exactly (\ref{ineq-scaling}) with $\omega=n\frac{\rho-1}{\rho}$. Now taking $v=0$, the previous estimate and inequality (\ref{L-est}) yields 
	\begin{align}
	\Vert \Psi(u)\Vert_{{X}^{1,\infty,\infty}_{rq}(\mathbb{R}^{n+1}_+)} &\leq \Vert \mathbf{S}f\Vert_{{X}^{1,\infty,\infty}_{rq}(\mathbb{R}^{n+1}_+)} + \Vert \Psi(u)-\Psi(0)\Vert_{{X}^{1,\infty,\infty}_{rq}(\mathbb{R}^{n+1}_+)}\nonumber\\
	&\leq C \Vert f\Vert_{\mathcal{M}_{p\infty}^{\omega}(\mathbb{R}^{n})} +\Big( 2C\Big(\frac{2\varepsilon}{1-L}\Big)^{\rho-1} +L\Big)\Vert u\Vert_{{X}^{1,\infty,\infty}_{rq}(\mathbb{R}^{n+1}_+)}\nonumber\\
	&\leq \varepsilon + \Big( 2C\Big(\frac{2\varepsilon}{1-L}\Big)^{\rho-1} +L\Big) \frac{2\varepsilon}{1-L}\nonumber\\
	&< \frac{2\varepsilon}{1-L},\nonumber
	\end{align}
where $\Vert f\Vert_{\mathcal{M}_{p\infty}^{\omega}(\mathbb{R}^{n})}\leq \varepsilon/C$ for  $\varepsilon>0$ such that 
$ 2C\Big(\frac{2\varepsilon}{1-L}\Big)^{\rho-1}<\frac{1-L}{2}.$ Therefore,  the map $u\mapsto \Psi(u)$ has a unique fixed point $u\in B_{\varepsilon}\subset {X}^{1,\infty,\infty}_{rq}(\mathbb{R}^{n+1}_+)$ which is {solution} of the integral equation \eqref{key-eq-solve} as we wish to show. \fin

\bigskip 
\noindent\textbf{Proof of Theorem \ref{Symmetries}}: The solution $u$ of  Theorem
\ref{existence} can be seen as a limit in ${X}^{1,\infty,\infty}_{rq}(\mathbb{R}^{n+1}_+)$ of the following Picard sequence:
\begin{equation}\label{Picard}
 u_{1}=\mathbf{S}(f), \,\,\,\,\,\, u_{k+1}=u_1+\mathcal{T}_{V}(u_k)+\mathcal{B}(u_k), \,\,\,\, k\in \mathbb{N}.
\end{equation}
Since $f$ is non-negative in $\partial \mathbb{R}^{n+1}_{+}$ and positive in the measurable set $\mathcal{D}$, then 
\begin{equation}
u_{1}(x,t)=\int_{\partial \mathbb{R}^{n+1}_{+}}G(x-y,t)f(y)dy \,>0 \, \text{ in } \mathbb{R}^{n+1}_{+}.
\end{equation}
Moreover, if $V$ is non-negative on $\partial \mathbb{R}^{n+1}_{+}$ we conclude that $\mathbf{S}(Vu)+\mathbf{S}(b\vert u \vert^{\rho-1}u)$ is also non-negative on $\partial \mathbb{R}^{n+1}_{+}$ whenever that $u$ restricted to $\partial \mathbb{R}^{n+1}_{+}$ is non-negative. By an induction argument one can prove that each element of the sequence $\{u_k\}_k$ is positive.
Note that convergence of $\{u_k\}_k$ in ${X}^{1,\infty,\infty}_{rq}(\mathbb{R}^{n+1}_+)$ imply  $u_k\xrightarrow{k\rightarrow\infty} u$ in $\mathcal{M}_{q\infty}^{\lambda}(\partial \mathbb{R}^{n+1}_{+})$ and $u_k\xrightarrow{k\rightarrow\infty} u$ in $\dot{W}^{1}\mathcal{M}_{r\infty}^{\mu}(\mathbb{R}^{n+1}_+)$, the homogeneous Sobolev space based in weak-Morrey. From the first convergence we get convergence in measure on $\partial \mathbb{R}^{n+1}_{+}$ and from the second (see Theorem \ref{HLS}) we get convergence in measure on  $\mathbb{R}^{n+1}_{+}$. Therefore, there is a subsequence $\{u_{k_j}\}_j$ which converges pointwise to $u$  except a null set in $(\partial \mathbb{R}^{n+1}_{+}, dx)$ or  $(\mathbb{R}_{+}^{n+1}, dxdt)$. Since  $u_{k_j}(x,t)>0$, we conclude that $u(x,t)$ is non-negative a.e. in $\mathbb{R}^{n+1}_+$, but $u$ is \textit{solution} of the integral equation (\ref{key-eq-solve}), then $u=u_1+\mathcal{T}_V(u)+\mathcal{B}(u)\geq u_1>0$ as we wish to show.

\end{document}